\documentclass{article}
\usepackage{graphicx}% Include figure files
\usepackage{physics}
\usepackage{caption}
\usepackage{subcaption}
\usepackage[export]{adjustbox}
\usepackage{amsmath,bm}
\usepackage{amssymb}
\usepackage{mathtools}
\usepackage{xparse}
\usepackage{todonotes}
\usepackage{amsthm}
\usepackage{soul}
\usepackage{amsmath}
\DeclarePairedDelimiterX{\inp}[2]{\langle}{\rangle}{#1, #2} 
\usepackage{acronym}
\usepackage[capitalise]{cleveref}

\newcommand{\REPHRASE}[1]{{\todo[inline, color=yellow!40]{REPHRASE: #1}}}
\newcommand{\inote}[1]{{\todo[inline, color=blue!40]{Ido: #1}}}

\newcommand{\OFF}[1]{}

\newtheorem{theorem}{Theorem}[section]
\newtheorem{lemma}[theorem]{Lemma}
\newtheorem{corollary}[theorem]{Corollary}
\newtheorem{definition}[theorem]{Definition}
\newtheorem{proposition}[theorem]{Proposition}

\newtheorem{remark}[theorem]{Remark}

\title{Functional Dimensionality of Koopman Eigenfunction Space}
\author{Ido Cohen$^1$ \\ idocoh@ariel.ac.il  \and Eli Appleboim$^2$ \\ eliap@ee.technion.ac.il \and Gershon Wolansky$^3$ \\ gershonw@math.technion.ac.il}
\date{%
    $^1$Electrical and Electronics Engineering, Ariel University, Ariel\\%
    $^2$Electrical and Computer Engineering, Technion, Haifa\\
    $^3$Mathematics, Technion, Haifa\\[2ex]%
    \today
}

\begin{document}
% 
%\dedicatory{This paper is dedicated to MD Naama Cohen.}

\maketitle
\begin{abstract}
This work presents the general form solution of \acl{KPDE} and shows that its functional dimensionality is finite. The dimensionality is as the dimensionality of the dynamics. Thus, the representation of nonlinear dynamics as a linear one with a finite set of Koopman eigenfunctions without error is possible. This formulation justifies the flowbox statement and provides a simple numerical method to find such representation.
\end{abstract}

{\bf{Keywords:}} Koopman operator, Koopman Partial Differential Equation, Dynamical systems, Flowbox, Conservation laws

\section{Introduction}
The Koopman spectrum is a commonly used tool for dynamical system analysis. Treating Koopman Eigenfunction space as an infinite dimensional vector space yields various techniques to represent the system as a linear one with truncated dimensionality \cite{schmid2010dynamic, brunton2016sparse, kaiser2018discovering, servadio2023koopman, schmid2022dynamic}. Naturally, these methods occasionally result in an overly redundant spectral decomposition \cite{mezic2005spectral, williams2015data}, which is often inaccurate \cite{cohen2021modes, cohen_gilboa_2023}. Thus, the challenge of extracting meaningful information about the dynamics from samples remains open \cite{avila2020data}. In this study, we present the general solution of \ac{KPDE} and show that only $N$ Koopman Eigenfunctions are required for perfectly representing a nonlinear dynamic as a linear one.

Regrettably, the mathematical framework of the Koopman spectrum has received little attention, resulting in a limited understanding of how to efficiently extract a representation based on the underlying geometry of this space from samples \cite{bollt2021geometric}. This lack of knowledge has led unsophisticated and exhaustive algorithms \cite{brunton2022data, williams2015data, li2017extended} that led to intrinsic flaws in dynamical representation and prediction \cite{wu2021challenges}, for example, in highly nonlinear time-variant systems \cite{turjeman2022underlying}, homogeneous flows \cite{cohen2021modes}, or even linear systems with non-zero inputs \cite{lu2021extended}. These drawbacks led to ad-hok adaptations to almost every application \cite{schmid2022dynamic}. This variety of adaptations challenges the liability of the algorithms due to their lack of stability \cite{wynn2013optimal}, accuracy \cite{hemati2017biasing, dawson2016characterizing}, and of robustness to noise \cite{kutz2016dynamic, bagheri2013effects, nonomura2018dynamic, nonomura2019extended, azencot2019consistent}.
This study aims to bridge this knowledge gap by presenting a comprehensive theory on the general form of the \ac{KPDE}.
\paragraph{Main contributions}

\begin{itemize}
    \item The general form solution of \ac{KEF} is presented.
    \item The general solution of \ac{KPDE} is formulated.
    \item A minimal set of Koopman eigenfunction space \cite{cohen2023minimal} is naturally stemmed from \ac{KPDE}'s solution.
    \item A flowbox representation and  \ac{KPDE}'s solution overlap each other.
    \item The functional dimensionality is as the dimensionality of the dynamics.
\end{itemize}

\paragraph{Paper Outlines}
The plan of this paper is as follows. The general solution of \ac{KPDE} is formulated based on the characteristics method. Consequently, the functional dimensionality of \ac{KEF} space is proven to be finite. Namely, we need only $N$ Koopman eigenfunction to represent a dynamical system as a linear one. At last, we show a direct connection between the minimal set concept, flowbox theorem, and conservation laws of dynamical systems. We precede with preliminary definitions and identities.
%In the numerical part, based on the geometry induced from their functional dimensionality, the minimal set of \ac{KEF} is found analytically and numerically. We suggest a method to represent a nonlinear dynamical system as a linear one with only $N$ eigenfunctions (\cref{set:numericalPart}). Discussion about the pros and cons of this method and its advantages over other methods is brought in \cref{sec:prosCons}.

\section{Preliminaries}
We list below essential notations and definitions that are used in this paper.
\begin{description}%[leftmargin=*]
    \item[Dynamic]  Let us consider the following nonlinear dynamical system, defined in a domain $\mathcal{D}$ in $\mathbb{R}^n$
\begin{equation}\label{eq:dynamics}
    \dot{\bm{x}}=P(\bm{x}), \,\, t\in I=[0,T]
\end{equation}
where $\bm{x}\in \mathbb{R}^N$, the operator $\,\dot{\,\,}\,$ denotes the time derivative, and $P:\mathbb{R}^N\to \mathbb{R}^N$. All along this work it is assumed that $P$ is $C^1$ and therefore $\bm{x}(t)\in C^2$.

\OFF{\item[Orbit of an initial point] Given an initial condition, $\bm{x}(t=0)= \bm{x}_0$, the unique solution of \eqref{eq:dynamics} can be seen as a curve in $\mathbb{R}^N$. This trajectory is denoted by $\mathcal{X}(\bm{x}_0)$, and termed as the \emph{orbit of $\bm{x}_0$}.}

\item[Equilibrium] An equilibrium point, denoted by $\bm{x}^*\in \mathcal{D}$, is a stationary point of Eq.  \eqref{eq:dynamics}, i.e. a point at which
\begin{equation}
    P(\bm{x}^*)=\bm{0}
\end{equation}
where $\bm{0}$ is the $N$ dimensional zero vector.

\item[Measurement] A measurement is a function from $\mathcal{D}$ to $\mathbb{C}$. 

\item[Koopman Operator] The Koopman operator $K_\tau$ acts on the infinite dimensional vector space of measurements and admits the following.  Let $g(\bm{x})$ be a measurement then
\begin{equation}\label{eq:koopDef}\tag{\bf{KO}}
    K_\tau(g(\bm{x}(s)))=g(\bm{x}(s+\tau)),\quad s,s+\tau\in I,
\end{equation}
where $\tau>0$. This operator is linear \cite{koopman1931hamiltonian,mezic2005spectral}.
\item[\acl{KEF}] Assuming the initial condition $\bm{x}_0$, a measurement $\varphi(\bm{x})$, satisfying the following relation along the orbit $\mathcal{X}(\bm{x}_0)$
\begin{equation}\label{eq:KEFdiff}
    \dfrac{d\varphi (\bm{x})}{dt}=\lambda \varphi(\bm{x}), \quad \forall \bm{x}\in \mathcal{X}(\bm{x}_0)
\end{equation}
for some value $\lambda\in \mathbb{C}$, is a \acf{KEF}.
\item[Koopman \acs{PDE}] Let $\Phi(\bm{x})$ be some differentiable measurement. Then $\Phi$ is a solution of the Koopman \ac{PDE} if it satisfies the following, everywhere in $\mathcal{D}$,
\begin{equation}\label{eq:KEFPDE}
    \nabla \Phi (\bm{x})^T P(\bm{x}) = \lambda \Phi(\bm{x}), \quad \forall \bm{x}\in \mathcal{D}. 
\end{equation}
where $\nabla$ denotes the gradient of $\Phi$ with respect to the state vector $\bm{x}$. In particular, it is assumed that $\Phi$ is $C^1$ as a function of $\bm{x}$. 

\item[Conservation Law] The function $h:\mathcal{D}\to \mathbb{R}$ is a conservation law if it is a (non constant) solution to \cref{eq:KEFPDE} associated with $\lambda =0$.

\item[Unit Measurement] A unit velocity measurement is a smooth function $m:\mathcal{D}\subset \mathbb{R}^N \to \mathbb{C}$  satisfying the following \ac{PDE} \footnote{This function is denoted with $m$ as a shortage for $\mu o v \acute\alpha \delta \alpha$ unit in Greek.},
\begin{equation}\label{eq:unitPDE}
    \nabla m(\bm{x})^T P(\bm{x})=1,\quad \forall \bm{x}\in \mathcal{D}.
\end{equation}
The change rate of the unit measurement, dictated by the dynamics, is given by
\begin{equation}\label{unit}
    \frac{d}{dt}m(\bm{x})=\nabla m(\bm{x})^T P(\bm{x})=1, 
\end{equation}
which is the source of its name.

\end{description}

\section{Solution of \ac{KPDE}, Minimal Set, and Flowbox}\label{sec:theoreticPart}

\subsection{General Solution of \ac{KPDE}}
Let $S$ be an open, $N-1$ dimensional hyper surface  embedded in $\mathbb{R}^N$. We say that $S$ is {\em non-recurrent} with respect to $P$ if any solution orbit  of $P$ intersect $S$ {\em at  most} once. 
\subsubsection{Examples}\label{examples}
Let $N=2$
\begin{description}
    \item{[a]} $P:= (x_1, x_2)$.

 We can choose $S$ to be any circle $x_1^2+x_2^2=R>0$, excluding one point $(x_1^0, x_2^0)$ on this circle. 
Then $\Omega(S)$ is the entire  plan $\mathbb{R}^2$, excluding the orbit passing trough this point.
\item{[b]} $P:= (-x_1, x_2)$.

We can define $S$ to be  $x_1=1$.   In that case $\Omega(S)$ is the half plain $x_1>0$. 
\item{[c]} If  $P:= (x_2, -x_1)$.  Then all orbits are composed of the circles  $x_1^2+x_2^2=R$, so any point on any surface is recurrent.  
\end{description}

Let $\Omega(S)\subset \mathbb{R}^N$  be the set of all points on the solution orbits intersecting $S$.
The characteristic method for Cauchy  Problem applied to  linear, first order PDE implies (c.f. Courant \& Hilbert \cite{hilbert1985methods}) 
\begin{theorem}
   Given a function $g$ on $S$,  the equation 
    $$ \nabla\Psi^T(x) \cdot P(x)=h(x)$$
    has a unique solution in the domain $\Omega(S)$
satisfying $\Psi=g$ on $S$.
\end{theorem}

Using this Theorem for $g=0$ and $h=1$, we now define $m$ to be  the {\em unique} solution of (\ref{eq:unitPDE}) on $\Omega(S)$ satisfying $m(x)\equiv 0$ on $S$. 
\begin{remark}
The existence of such unit measurement $m$ on $\Omega(S)$ follows from the assumption that $S$ is non-recurrent.
\end{remark}

To formulate the general solution of Koopman PDE, we first state the following Lemma.
\begin{lemma}\label{Lem:UnitVelocityKEF}
    
    Any $\lambda-$Koopman eigenfunction on $\Omega(S)$ is of the form $h(\bm{x})e^{\lambda m(\bm{x})}$ where $h$ is an invariant of $P$, on $\Omega(S)$, namely
    \begin{equation}
        \nabla^T h(\bm{x})\cdot P(\bm{x})=0 \ \ \forall x\in\Omega(S)
    \end{equation} 
\end{lemma}
\begin{proof}
Let $\Phi(\bm{x})$ be a $\lambda$-Koopman eigenfunction. Then 

\begin{equation}
    \begin{split}
        &\nabla^T\left(\Phi(\bm{x}) e^{-\lambda m(\bm{x})}\right)\cdot P=\\
        &\qquad e^{-\lambda m(\bm{x})} \underbrace{\nabla^T \Phi(\bm{x})\cdot P}_{=\lambda \Phi(\bm{x})}-e^{-\lambda m(\bm{x})}\Phi(\bm{x}) \lambda \underbrace{\nabla^Tm(\bm{x})\cdot P}_{=1}=0
    \end{split}
\end{equation}
using (\cref{eq:KEFPDE}, \cref{unit}). 
\end{proof}
\begin{theorem}[General solution of \acs{KPDE}] \label{theo:generalSol}
There exists $N-1$ invariant functions $h_i:\Omega(S)\rightarrow (0,1)$  such that any Koopman eigenfunction defined on $\Omega(S)$ is of the form 
\begin{equation}\label{eq:generalSolutionKPDE}
    \Phi(\bm{x})=f(h_1(\bm{x}),\cdots,h_{N-1}(\bm{x}))e^{\lambda M(\bm{x})}
\end{equation}
where $f$ is any differentiable function on $(0,1)^{N-1}$. 
\end{theorem}

\begin{proof}
We first show the existence of $N-1$ invariants $h_1, \ldots h_{N-1}$ of  on $\Omega(S)$  such that  any invariant on $\Omega(S)$ is of the form $f(h_1, \ldots h_{N-1})$. 
    
 Since $S$ is  homeomorphic to the $N-1$ dimensional ball, 
  %\enote{Is that true? Couldn`t it be a sphere?}
  there exists $N$ functions $X:=s_1, \ldots s_{N}$ on  $[0,1]^{N-1}$ which are parameterization of $S$, namely $X:(0,1)^{N-1}\rightarrow S$ is a $1-1$ surjection. 

Let $h_i$ be the unique  solution of the Caushy problem $\nabla h_i^T\cdot P(x)=0$ on $\Omega(S)$ such that $h_i(X(\tau_1, \ldots, \tau_{N-1}))=\tau_i$ for any $(\tau_1, \ldots \tau_{N-1})\in (0,1)^{N-1}$. 
  
  Claim: Any invariant of the system \cref{eq:dynamics} is given by $f(h_1(\bm{x}), \ldots h_{N-1}(\bm{x}))$ for some differentiable function $f$ on $(0,1)^{N-1}$. Indeed, suppose $G$ is some invariant function on $\Omega(S)$. Let $g$ be the restriction of $G$ to $S$. Since $X$ is a parameterization of $S$ on $(0,1)^{N-1}$ then $g\circ X=f $ for some $f$ defined on $(0,1)^{N-1}$. 

    Since $G$ is an invariant on $\Omega(S)$, it is the {\em unique} solution of the Cauchy problem  $\nabla^T \Phi\cdot P=0 $ such that $\Phi=g$ on $S$. On the other hand, the function $g$ equals $f(h_1, \ldots h_{N-1})$ on $S$  by definition. Moreover, 
    $$\nabla^T f(h_1, \ldots h_{N-1}) \cdot P= \sum_{i=1}^{N-1}\frac{\partial f}{\partial h_i} \nabla^T h_i\cdot P =0$$
    so $f(h_1, \ldots h_{N-1})=G$ by uniqueness. 
    
    The Theorem now follows from \cref{Lem:UnitVelocityKEF}. 
\end{proof}

Back to the Examples \ref{examples}: In case [a] we get $M=\frac{1}{2}\ln (x_1^2+ x_2^2)$. A possible invariant is $h(x_1, x_2)=x_1/x_2$. This invariant is not defined on  the line $x_2=0$,  but we may define $\tilde{h}=\arctan\left(x_1/ x_2\right)\equiv \arg(x_1, x_2)$ as an invariant. Depending  the chosen branch of the argument, $\tilde{h}$ can be extended on the entire plain, excluding the orbit of a single solution intersecting a given point $(x^0_1, x^0_2)$ on the circle. The Koopman eigenfunction 
$\Phi_\lambda=
\arg(x_1, x_2) e^{\frac{\lambda(x_1^2+x_2^2)}{2}}$ 
however, {\em cannot} be extended to the entire plain (see \cref{fig:initSure}). 
\begin{figure}[phtb!]
    \centering
    \captionsetup[subfigure]{justification=centering}
    \begin{subfigure}[t]{0.48\textwidth} %{0.25\textwidth}
\includegraphics[width=1\textwidth,valign = t]{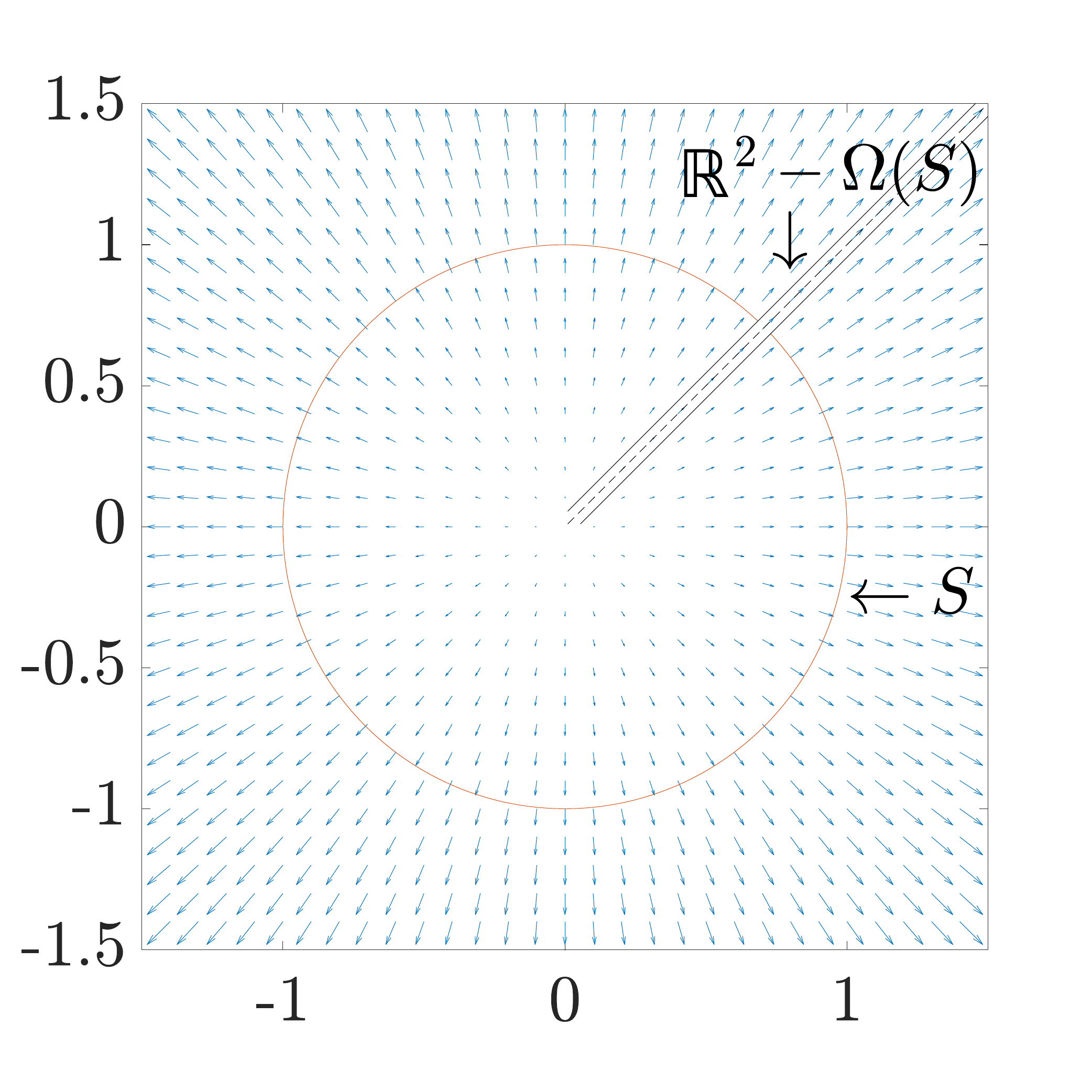}
        \label{subfig:attractor}
    \end{subfigure}
    \begin{subfigure}[t]{0.48\textwidth} %{0.25\textwidth}
\includegraphics[width=1\textwidth,valign = t]{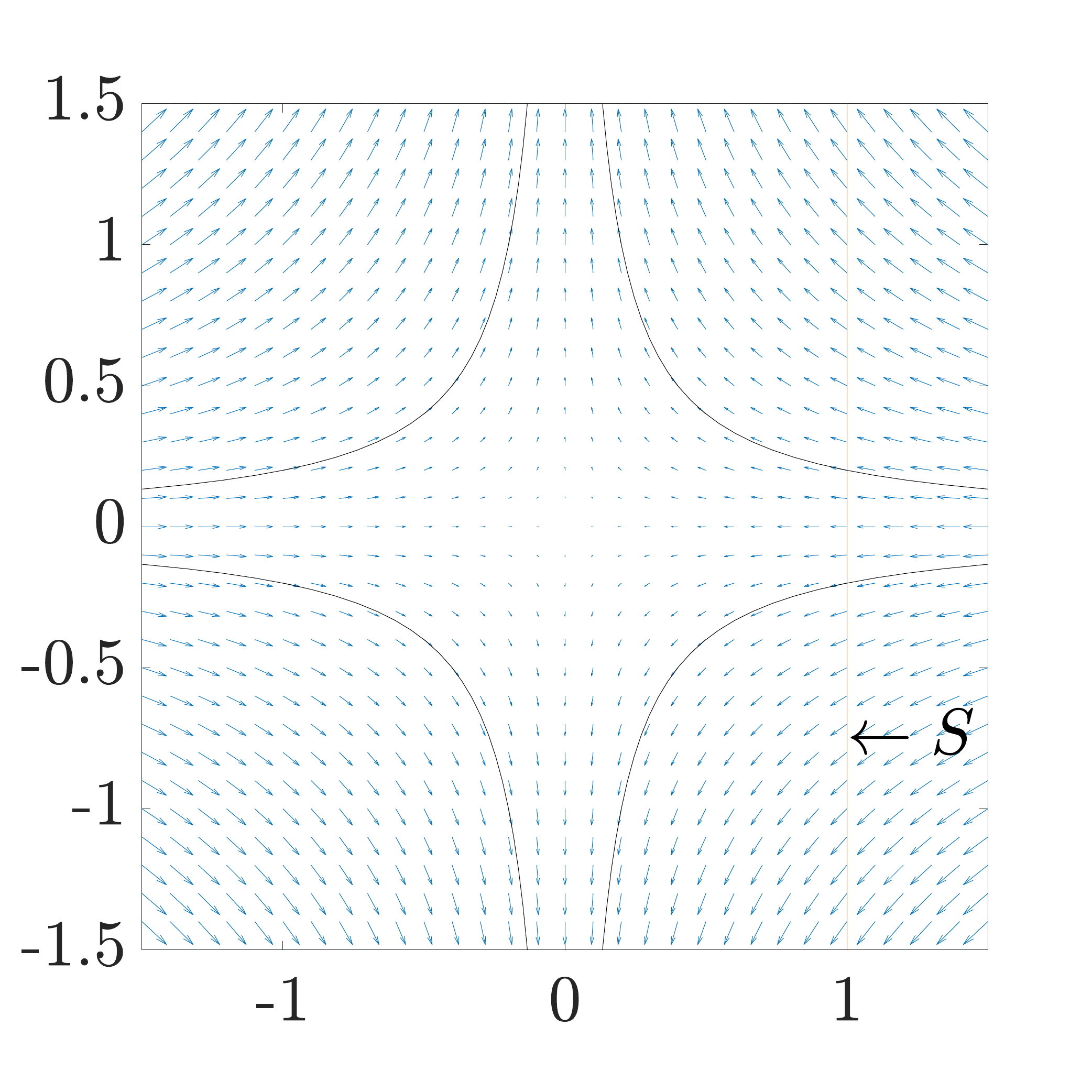}
        \label{subfig:hyperbolic}
    \end{subfigure}
        
        \caption{Initial surfaces for a source and a hyperbolic systems}    
    \label{fig:initSure}
    % \label{fig:2DLinReal}
\end{figure}

In case [b], $M=\frac{1}{2}\ln\left(\frac{1}{x^2_1}\right)$ on the half plain defined by $x_1>0$ (consistent with $M(1, x_2)=0$). We obtain an invariant $h(x_1, x_2)=x_1x_2$. This invariant,  however,  can be extended to the entire plain. 

 In case [c] there are no non-recurrent surfaces, so there are no Koopman eigenfunctions for arbitrary eigenvalue $\lambda$. However, in this case there exists eigenfunctions of the form $h(x_1^2+x_2^2) e^{in\arg(x_1, x_2)}$ for any function $h$ and integer $n$, corresponding to $\lambda=in$.

\OFF{Following \cref{theo:generalSol}, we can define a minimal as,
\begin{definition}[Minimal Set]\label{def:minimalSet} A minimal set of Koopman eigenfunctions of an $N$ dimensional dynamical system is a set of $N$ eigenfunctions that their Jaacobian matrix is of full rank.
\end{definition}}

\subsection{Minimal Set}
The mathematical structure of \ac{KEF} space is defined and studied thoroughly in \cite{cohen2023minimal}. This structure is not a ring but more complex than a group. This structure is algebraic-differential, meaning there are admissible actions following differential conditions for which the Koopman space is closed. After defining this structure, the author defined dependent and independent sets of \acp{KEF}. A set of \acp{KEF} are independent if their Jaacobian is a full-rank matrix. Thus, there is a set of at maximum $N$ independent \acp{KEF}, named a \emph{minimal set}, from which one can generate the rest of the eigenfunctions. Following \cref{theo:generalSol}, a minimal set can be,
\begin{equation}\label{eq:minimalSet}
  \begin{split}
        \Phi_1(\bm{x})&=h_1(\bm{x})e^{m(\bm{x})}\\
        \Phi_2(\bm{x})&=h_2(\bm{x})e^{m(\bm{x})}\\
        &\,\,\,\vdots\\
        \Phi_{N-1}(\bm{x})&=h_{N-1}(\bm{x})e^{m(\bm{x})}\\
        \Phi_N(\bm{x})&=e^{m(\bm{x})}
    \end{split}.
\end{equation}

\subsection{Flowbox}
%\gnote{abc}
Flowbox is a coordinate system in which the dynamic is trivial. Meaning, the velocity of one coordinate is one and the rest's are zeros. The Flowbox theorem states that for any point in a Lipschitz vector field there is an invertible transformation from a  neighborhood of a point to a flowboxed coordinate \cite{calcaterraboldt2008flowbox}. This statement holds if the point is far from a system singularity. \cref{theo:generalSol} naturally induces  a flowbox coordinate system,
\begin{equation}\label{eq:flowbox}
    \begin{split}
        z_1 &= h_1(\bm{x})\\
        z_2 &= h_2(\bm{x})\\
        &\,\,\,\vdots\\
        z_{N-1} &= h_{N-1}(\bm{x})\\
        z_N &=m(\bm{x})
    \end{split}.
\end{equation}
and, in these coordinates, any solution satisfying $\bm{x}(0)\in S$  takes the form 
\begin{equation}\label{eq:canonicalDynamicSplitting}
    \begin{split}
        z_1(t)&=h_1(\bm{x}(0))\\
        &\,\,\,\vdots\\
        z_{N-1}(t)&= h_{N-1}({\bm{x}(0)})\\
        z_N(t)&=t
    \end{split}.
\end{equation}
for any $-\infty < t < \infty$.

\section{Conclusions}
In many applications in signal and image processing, data mining we are looking for a good representation. Often, the measure of "goodness" depends on the space in which data is embedded and the specific application at hand. However, generally speaking, sparsity and accuracy are the crucial parameter for data representation, reconstructing, storing and retrieval. 

The same considerations are guiding us in dynamical system representations which lead us to formulate the solutions of \ac{KPDE} concisely and accurately. By using the characteristic method, one can formulate the general solution of \ac{KPDE} as a function of $N$ independent Koopman Eigenfunctions. This formulation emphasizes the importance of the geometry considerations in finding \acp{KEF} numerically for "good" representation. Thus, randomly sampled vector field or sampled orbits randomly initiated can easily reveal the underlying dynamic under the assumption of the dynamic and \acp{KEF}' smoothness.

\section*{List of Acronyms}
\begin{acronym}
\acro{KEF}[KEF]{\emph{Koopman Eigenfunction}}
\acro{KPDE}[KPDE]{\emph{Koopman Partial Differential Equation}}
%\acro{FTFC}[FTFC]{\emph{Fundamental Theorem of Fractional Calculus}}
% \acro{KMD}[KMD]{\emph{Koopman Mode Decomposition}}
%\acro{KEFal}[KEFal]{\emph{Koopman Eigenfunctional}}
% \acro{DMD}[DMD]{\emph{Dynamic Mode Decomposition}}
%\acro{SDMD}[S-DMD]{\emph{Symmetric DMD}}
%\acro{EDMD}[EDMD]{\emph{Extended DMD}}
%\acro{KDMD}[KDMD]{\emph{kernel DMD}}
%\acro{SVD}[SVD]{\emph{Singular Value Decomposition}}
%\acro{TV}[TV]{\emph{Total Variation}}
%\acro{DFT}[DFT]{\emph{Discrete Fourier Transform}}
%\acro{CFT}[CFT]{\emph{Continuous Fourier Transform}}
%\acro{POD}[POD]{\emph{Proper Orthogonal Decomposition}}
%\acro{ROA}[ROA]{\emph{Region of Attraction}}
\acro{PDE}[PDE]{\emph{Partial Differential Equation}}
%\acro{TV}[TV]{\emph{Total Variation}}
%\acro{OrthoNS}[OrthoNS]{\emph{Orthogonal Nonlinear Spectral decomposition}}

\end{acronym}

\bibliographystyle{amsplain}
\bibliography{smartPeople}

\providecommand{\bysame}{\leavevmode\hbox to3em{\hrulefill}\thinspace}
\providecommand{\MR}{\relax\ifhmode\unskip\space\fi MR }
% \MRhref is called by the amsart/book/proc definition of \MR.
\providecommand{\MRhref}[2]{%
  \href{http://www.ams.org/mathscinet-getitem?mr=#1}{#2}
}
\providecommand{\href}[2]{#2}
\begin{thebibliography}{10}

\bibitem{avila2020data}
Allan~M Avila and I~Mezi{\'c}, \emph{Data-driven analysis and forecasting of
  highway traffic dynamics}, Nature communications \textbf{11} (2020), no.~1,
  1--16.

\bibitem{azencot2019consistent}
Omri Azencot, Wotao Yin, and Andrea Bertozzi, \emph{Consistent dynamic mode
  decomposition}, SIAM Journal on Applied Dynamical Systems \textbf{18} (2019),
  no.~3, 1565--1585.

\bibitem{bagheri2013effects}
Shervin Bagheri, \emph{Effects of small noise on the {DMD}/{K}oopman spectrum},
  Bulletin Am. Phys. Soc \textbf{58} (2013), no.~18, H35.

\bibitem{bollt2021geometric}
Erik~M Bollt, \emph{Geometric considerations of a good dictionary for koopman
  analysis of dynamical systems: Cardinality,“primary eigenfunction,” and
  efficient representation}, Communications in Nonlinear Science and Numerical
  Simulation \textbf{100} (2021), 105833.

\bibitem{brunton2022data}
Steven~L Brunton and J~Nathan Kutz, \emph{Data-driven science and engineering:
  Machine learning, dynamical systems, and control}, Cambridge University
  Press, 2022.

\bibitem{brunton2016sparse}
Steven~L Brunton, Joshua~L Proctor, and J~Nathan Kutz, \emph{Sparse
  identification of nonlinear dynamics with control (sindyc)},
  IFAC-PapersOnLine \textbf{49} (2016), no.~18, 710--715.

\bibitem{calcaterraboldt2008flowbox}
C.~Calcaterra and A.~Boldt, \emph{Lipschitz flow-box theorem}, J. Math. Anal.
  Appl. \textbf{338} (2008), 1108 -- 1115.

\bibitem{cohen2023minimal}
Ido Cohen and Eli Appelboim, \emph{A minimal set of koopman
  eigenfuncitons--analysis and numerics}, arXiv preprint arXiv:2303.05837
  (2023).

\bibitem{cohen2021modes}
Ido Cohen, Omri Azencot, Pavel Lifshits, and Guy Gilboa, \emph{Modes of
  homogeneous gradient flows}, SIAM Journal on Imaging Sciences \textbf{14}
  (2021), no.~3, 913--945.

\bibitem{cohen_gilboa_2023}
Ido Cohen and Guy Gilboa, \emph{Latent modes of nonlinear flows: A koopman
  theory analysis}, Elements in Non-local Data Interactions: Foundations and
  Applications, Cambridge University Press, 2023.

\bibitem{dawson2016characterizing}
Scott~TM Dawson, Maziar~S Hemati, Matthew~O Williams, and Clarence~W Rowley,
  \emph{Characterizing and correcting for the effect of sensor noise in the
  dynamic mode decomposition}, Experiments in Fluids \textbf{57} (2016), no.~3,
  42.

\bibitem{hemati2017biasing}
Maziar~S Hemati, Clarence~W Rowley, Eric~A Deem, and Louis~N Cattafesta,
  \emph{De-biasing the dynamic mode decomposition for applied {K}oopman
  spectral analysis of noisy datasets}, Theoretical and Computational Fluid
  Dynamics \textbf{31} (2017), no.~4, 349--368.

\bibitem{hilbert1985methods}
David Hilbert, \emph{Methods of mathematical physics}, CUP Archive, 1985.

\bibitem{kaiser2018discovering}
Eurika Kaiser, J~Nathan Kutz, and Steven~L Brunton, \emph{Discovering
  conservation laws from data for control}, 2018 IEEE Conference on Decision
  and Control (CDC), IEEE, 2018, pp.~6415--6421.

\bibitem{koopman1931hamiltonian}
Bernard~O Koopman, \emph{Hamiltonian systems and transformation in hilbert
  space}, Proceedings of the national academy of sciences of the united states
  of america \textbf{17} (1931), no.~5, 315.

\bibitem{kutz2016dynamic}
J~Nathan Kutz, Steven~L Brunton, Bingni~W Brunton, and Joshua~L Proctor,
  \emph{Dynamic mode decomposition: data-driven modeling of complex systems},
  pp.~119--132, SIAM, 2016.

\bibitem{li2017extended}
Qianxiao Li, Felix Dietrich, Erik~M Bollt, and Ioannis~G Kevrekidis,
  \emph{Extended dynamic mode decomposition with dictionary learning: A
  data-driven adaptive spectral decomposition of the koopman operator}, Chaos:
  An Interdisciplinary Journal of Nonlinear Science \textbf{27} (2017), no.~10,
  103111.

\bibitem{lu2021extended}
Hannah Lu and Daniel~M Tartakovsky, \emph{Extended dynamic mode decomposition
  for inhomogeneous problems}, Journal of Computational Physics \textbf{444}
  (2021), 110550.

\bibitem{mezic2005spectral}
Igor Mezi{\'c}, \emph{Spectral properties of dynamical systems, model reduction
  and decompositions}, Nonlinear Dynamics \textbf{41} (2005), no.~1, 309--325.

\bibitem{nonomura2018dynamic}
Taku Nonomura, Hisaichi Shibata, and Ryoji Takaki, \emph{Dynamic mode
  decomposition using a {K}alman filter for parameter estimation}, AIP Advances
  \textbf{8} (2018), no.~10, 105106.

\bibitem{nonomura2019extended}
\bysame, \emph{Extended-{K}alman-filter-based dynamic mode decomposition for
  simultaneous system identification and denoising}, PloS one \textbf{14}
  (2019), no.~2.

\bibitem{schmid2010dynamic}
Peter~J Schmid, \emph{Dynamic mode decomposition of numerical and experimental
  data}, Journal of fluid mechanics \textbf{656} (2010), 5--28.

\bibitem{schmid2022dynamic}
\bysame, \emph{Dynamic mode decomposition and its variants}, Annual Review of
  Fluid Mechanics \textbf{54} (2022), 225--254.

\bibitem{servadio2023koopman}
Simone Servadio, Roberto Armellin, and Richard Linares, \emph{A
  koopman-operator control optimization for relative motion in space}, AIAA
  SCITECH 2023 Forum, 2023, p.~0873.

\bibitem{turjeman2022underlying}
Rotem Turjeman, Tom Berkov, Ido Cohen, and Guy Gilboa, \emph{The underlying
  correlated dynamics in neural training}, arXiv preprint arXiv:2212.09040
  (2022).

\bibitem{williams2015data}
Matthew~O Williams, Ioannis~G Kevrekidis, and Clarence~W Rowley, \emph{A
  data--driven approximation of the koopman operator: Extending dynamic mode
  decomposition}, Journal of Nonlinear Science \textbf{25} (2015), 1307--1346.

\bibitem{wu2021challenges}
Ziyou Wu, Steven~L Brunton, and Shai Revzen, \emph{Challenges in dynamic mode
  decomposition}, Journal of the Royal Society Interface \textbf{18} (2021),
  no.~185, 20210686.

\bibitem{wynn2013optimal}
Andrew Wynn, DS~Pearson, Bharathram Ganapathisubramani, and Paul~J Goulart,
  \emph{Optimal mode decomposition for unsteady flows}, Journal of Fluid
  Mechanics \textbf{733} (2013), 473--503.

\end{thebibliography}
% \begin{thebibliography}{10}

% \bibitem {A} T. Aoki, \textit{Calcul exponentiel des op\'erateurs
% microdifferentiels d'ordre infini.} I, Ann. Inst. Fourier (Grenoble)
% \textbf{33} (1983), 227--250.

% \bibitem {B} R. Brown, \textit{On a conjecture of Dirichlet},
% Amer. Math. Soc., Providence, RI, 1993.

% \bibitem {D} R. A. DeVore, \textit{Approximation of functions},
% Proc. Sympos. Appl. Math., vol. 36,
% Amer. Math. Soc., Providence, RI, 1986, pp. 34--56.

% \end{thebibliography}

\end{document}